\documentclass[11pt, onecolumn]{IEEEtran}
\usepackage{amsthm}
\usepackage{amsmath}
\usepackage{times}
\usepackage{graphicx}
\usepackage{color}
\usepackage{multirow}
\usepackage{rotating}
\usepackage{bbm}
\usepackage{latexsym}
\usepackage{graphics, graphicx}
\usepackage{epstopdf}
\usepackage{algorithmic}
\usepackage{algorithm}
\usepackage{clrscode}
\usepackage[tight,footnotesize]{subfigure}
\usepackage[tight,footnotesize]{subfigure}
\usepackage{amsmath,epsfig,amssymb,mathrsfs}
\usepackage{dsfont}
\usepackage{amsmath}
\usepackage{mathcomp}
\usepackage{amsmath}
\usepackage{cite}

\DeclareMathOperator{\Tr}{Tr}

\newtheoremstyle{mytheoremstyle} 
    {\topsep}                    
    {\topsep}                    
    {}                   
    {}                           
    {\itshape} 
    {:}                          
    {.5em}                       
    {}  

\theoremstyle{mytheoremstyle}

\newtheorem{thm} {Theorem}

\newtheorem{lem}[thm]{Lemma}
\newtheorem{defn}[thm]{Definition}
\newtheorem{rem}[thm]{Remark}

\ifCLASSINFOpdf
\else
\fi

\begin{document}

\title{Analysis of Recurrent Linear Networks for Enabling Compressed Sensing of Time-Varying Signals}

\author{MohammadMehdi Kafashan,	Anirban Nandi, and~ShiNung Ching
\thanks{M. Kafashan and A. Nandi are with the Department of Electrical and Systems Engineering, Washington University in St. Louis, One Brookings Drive, Campus Box 1042, MO 63130, USA e-mail: kafashan@wustl.edu, nandia@ese.wustl.edu.}
\thanks{S. Ching is with Faculty of Electrical Engineering \& the Division of Biology and Biomedical Sciences, Washington
University in St. Louis, One Brookings Drive, Campus Box 1042, MO 63130, USA e-mail: shinung@ese.wustl.edu.}
}



\maketitle

\begin{abstract}
Recent interest has developed around the problem of dynamic compressed sensing, or the recovery of time-varying, sparse signals from limited observations.  In this paper, we study how the dynamics of recurrent networks, formulated as general dynamical systems, mediate the recoverability of such signals.  We specifically consider the problem of recovering a high-dimensional network input, over time, from observation of only a subset of the network states (i.e., the network output).  Our goal is to ascertain how the network dynamics lead to performance advantages, particularly in scenarios where both the input and output are corrupted by disturbance and noise, respectively.  For this scenario, we develop bounds on the recovery performance in terms of the dynamics.  Conditions for exact recovery in the absence of noise are also formulated.  Through several examples, we use the results to highlight how different network characteristics may trade off toward enabling dynamic compressed sensing and how such tradeoffs may manifest naturally in certain classes of neuronal networks.




\end{abstract}


\IEEEpeerreviewmaketitle

\section{Introduction}

\IEEEPARstart{W}{e} consider the analysis of recurrent networks for facilitating recovery of a high-dimensional, time-varying, sparse input in the presence of both corrupting disturbance and confounding noise. The network receives an input $u_t$ and generates the observations (network outputs), $y_t$ via its recurrent dynamics, i.e., 
\[
x_{t+1} = f(x_t,u_t,d_t)
\]
\[
y_t = g(x_t,e_t)
\]
where, here, $x_t$ are the network states, $d_t$ is the corrupting disturbance and $e_t$ is the confounding noise.  Our focus is on how the network dynamics, embedded in $f(\cdot), g(\cdot)$, impact the extent to which $u_t$ can be inferred from $y_t$ in the case where the dimensionality of the latter is substantially less than that of the former.  We will focus exclusively on the case where these dynamics are linear.

Such a problem, naturally, falls into the category of sparse signal recovery or compressed sensing (CS), for under-determined linear systems  \cite{candes2008introduction,eldar2012compressed,haupt2008compressed}. It is well known that for such problems, exact and stable recovery can be achieved under certain assumptions related to the statistical properties of the observed signal  \cite{donoho2006most,candes2006robust,candes2005decoding,candes2006stable}.  Classical CS, however, does not typically consider temporal dynamics associated with the recovery problem.

{
\subsection{Motivation}

Given the natural sparsity of electrical signals in the brain, CS has been linked to important questions in neural decoding \cite{petrantonakis2014compressed,wei2015bayesian}, i.e., how the brain represents and transforms information.   Understanding the dynamics of brain networks in the context of CS is a crucial aspect of this overall problem \cite{olshausen1996emergence,petrantonakis2014compressed, Barranca2014, barranca2014sparsity}. Such networks are, of course, not static.  Thus, recent interest has grown around so-called \textit{dynamic CS} and, specifically, on the recovery of signals subjected to transformation via a dynamical system (or, network).  In this context, sparsity has been formulated in three ways:  1) In the network states (state sparsity) \cite{vaswani2008kalman,charles2011sparsity,wakin2010observability}; 2) In the structure/parameters of the network model (model sparsity) \cite{sanandaji2011compressive,napoletani2008reconstructing}; and 3) In the inputs to the network (input sparsity) \cite{shoukry2014secure, fawzi2014secure,ba2012exact,sefati2015linear,jaeger2001short, white2004short, ganguli2008memory, hermans2010memory, wallace2013randomly, ganguli2010short,charles2014short}.  Here, we consider this latter category of recovery problems.

Our motivation is to understand how three stages of a generic network architecture -- an afferent stage, a recurrent stage, and an output stage (see Fig. \ref{fig:schematic1}) -- interplay in order to enable an observer, sampling the output, to recover the (sparse) input.  Such an architecture is pervasive in sensory networks in the brain wherein a large number of sensory neurons, receiving excitation from the periphery, impinge on an early recurrent network layer that transforms the afferent excitation en route to higher brain regions \cite{Ito2008a,Raman2010a}.    Moreover, beyond neural contexts, understanding network characteristics for dynamic CS may aid in the analysis of systems for efficient processing of naturally sparse time-varying signals \cite{ba2012exact}; and in the design of resilient cyber-physical systems \cite{shoukry2014secure, fawzi2014secure,sefati2015linear}, wherein extrinsic perturbations are sparse and time-varying.  Toward these potential instantiations, our specific aim in this paper is to elucidate fundamental dynamical {characteristics} of linear networks for exact and stable recovery of the (sparse) input signal, corrupted by an input disturbance.


\begin{figure}
	\centering
		\includegraphics[width=0.55\textwidth]{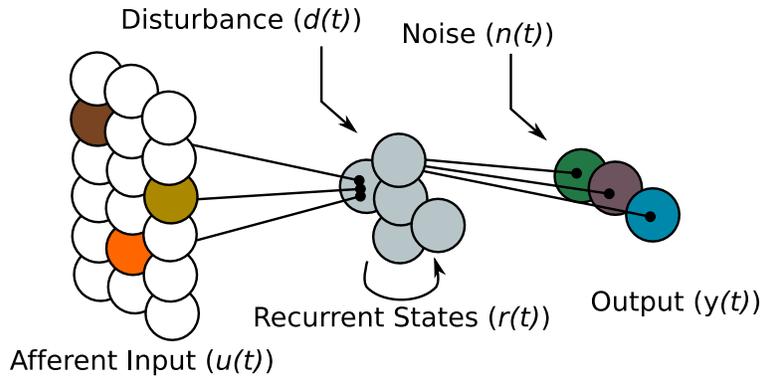}
	\caption{Schematic of the considered network architecture.  We study how the afferent, recurrent and output stages of this architecture interplay in order to enable accurate estimation of the input $u(t)$ from $y(t)$ in the presence of both disturbance and noise.}
	\label{fig:schematic1}
\end{figure}

{
\subsection{Paper Contributions}
} To achieve our specific aim, we develop and present the following contributions:
\begin{enumerate}
\item  We develop analytical conditions on the network dynamics, related to the classical notion of observability for a linear system, such that the network admits exact and stable (in the presence of output noise) input recovery.
\item  We derive an upper bound in terms of the network dynamics, for the $l_2$-norm of the reconstruction error over a finite time horizon.  This error can be defined in terms of both the disturbance and the noise.
\item  Based on the error analysis, we characterize a basic tradeoff between the ability of a network to simultaneously reject input disturbances while still admitting stable recovery.
\item We highlight network {characteristics} that optimally balance this tradeoff, and demonstrate via several examples their ability to reconstruct corrupted time-varying input from noisy observations.  In particular, we highlight an example of a rate-based neuronal network, and how specific features of the network architecture mediate these tradeoffs.
\end{enumerate}
}

{
\subsection{Prior Results in Sparse Input Recovery}
The sparse input recovery problem for linear systems can be formulated in both the spatial and temporal dimension.  Our contributions are related to the former.  As mentioned above, in this context, previous work has considered recovery of spatially sparse inputs for network resilience \cite{shoukry2014secure, fawzi2014secure} and encoding of inputs with sparse increments \cite{ba2012exact}.  In \cite{sefati2015linear}, conditions for exact sparse recovery are formulated in terms of a coherence-based observability criterion for intended applications in cyber-physical systems.  Our contributions herein provide a general set of analytical results, including performance bounds, pertaining to exact and stable sparse input recovery of linear systems in the presence of both noise and disturbance.

A second significant line of research in sparse input recovery problems pertains to the temporal dimension.  There, the goal is to understand how a temporally sparse signal (i.e., one that takes the value of zero over a nontrivial portion of its history) can be recovered from the state of the network at a particular instant in time.  This problem forms the underpinning of a characterization of `memory' in dynamical models of brain networks \cite{jaeger2001short, white2004short, ganguli2008memory, hermans2010memory, wallace2013randomly, ganguli2010short}.  In particular, in \cite{charles2014short} the problem of ascertaining memory is related to CS performed on the network states, over a receding horizon of a scalar-valued input signal.  In contrast to these works, we consider spatial sparsity of vector-valued inputs with explicit regard for both disturbance and an overt observation equation, i.e., states are not be directly sampled, but are transformed to an (in general lower-dimensional) output.

\subsection{Paper Outline}}
The remainder of the paper is organized as follows. In Section II we provide motivation of the current work and formulate the problem in detail. In Section III we develop theoretical results on the performance of the proposed recovery method. Simulation results for several different scenarios are provided in Section IV. Finally, conclusions are formulated in Section V.

\section{Problem Formulation}

We consider a discrete-time model for a linear network, formulated in the typical form a linear dynamical system, i.e.,:
\begin{equation}
\begin{aligned}
{\bf{r}}_{k+1} &= {\bf{A}} {\bf{r}}_k + {\bf{B}} {\bf{u}}_k + {\bf{d}}_k\\
{\bf{y}}_k &= {\bf{C}} {\bf{r}}_k + {\bf{e}}_k,
\end{aligned}
\label{linearDynamicFormulation}
\end{equation}
where $k$ is an integer time index, ${\bf{r}}_k\in\mathbb{R}{^{n}}$ is the activity of the network nodes (e.g., in the case of a rate-based neuronal network \cite{ostojic2009connectivity, kafashan2014bounded,abbott1999lapicque, dayan2005theoretical,izhikevich2003simple, kafashan2014node,kafashan2015optimal}, the firing rate of each neuron), ${\bf{u}}_k\in \mathbb{R}^m$ is the extrinsic input, ${\bf{d}}_k\in\mathbb{R}{^{n}}$ is the input disturbance, ${\bf{e}}_k\in\mathbb{R}{^{p}}$ is the measurement noise independent from ${\bf{d}}_k$, and ${\bf{y}}_k\in\mathbb{R}{^{p}}$ is the observation at time $k$. The matrix ${\bf{A}}\in\mathbb{R}{^{n\times n}}$ describes connections between nodes in the network, ${\bf{B}}\in\mathbb{R}{^{n\times m}}$ contains weights between input and output and ${\bf{C}}\in\mathbb{R}{^{p\times n}}$ is the measurement matrix. Such a model is, of course, quite general and can be used to describe recurrent dynamics in neuronal networks \cite{douglas2007recurrent,kohonen1976fast, seung2000stability,guler2005recurrent, omlin1996extraction}, machine learning applications such as pattern recognition and data mining \cite{ghanem2010sparse,wei2014adaptive,li2011fast,tao2004prediction}, etc.


We consider the case of bounded disturbance and noise, i.e.,  $\|{\bf{e}}_k\|_{\ell_2} \leq \epsilon$, $\|{\bf{d}}_k\|_{\ell_2} \leq \epsilon^{\prime}$. Since $m$, the number of input nodes, is larger than $n$, the number of output nodes, $\bf{B}$ takes the form of a  ``wide" matrix. We assume that at each time at most $s$ input nodes are active ($s$-sparse input), leading to an $\ell_0$ constraint to (\ref{linearDynamicFormulation}) at each time point:
\begin{equation}
\|{\bf{u}}_k\|_{\ell_0} \leq s.
\label{l0constraint}
\end{equation}
In the absence of disturbance and noise, recovering the input of \eqref{linearDynamicFormulation} with the $\ell_0$ constraint \eqref{l0constraint} amounts to the optimization problem:
\begin{equation}\label{P0}
\begin{aligned}
 & (P0)& \underset{\left({\bf{r}}_k\right)_{k=0}^K,\left({\bf{u}}_k\right)_{k=0}^{K-1}}{\text{minimize}}
 & & & \sum\limits_{k=0}^{K-1}\| {\bf{u}}_k \|_{\ell_0}\\
 & &\text{subject to}
 & & & {{\bf{r}}_{k + 1}} = {{\bf{A}}}{{\bf{r}}_k} + {{\bf{B}}}{{\bf{u}}_k}\\
 & &
 & & & {{\bf{y}}_k} = {\bf{C}}{{\bf{r}}_k}.
 \end{aligned}
\end{equation}
It is clear that Problem ($P0$) is a non-convex discontinuous problem, which is not numerically feasible and is NP-Hard in general \cite{muthukrishnan2005data}. For static cases, such $\ell_0$ optimization problems fall into the category of combinatorial optimization which require exhaustive search to find the solution \cite{parker1988}. 

Thus, throughout this paper, we follow the typical relaxation methodology used for such problems wherein the $\ell_0$ norm is relaxed to the $l_1$ norm, resulting in the problem:
\begin{equation}\label{P1}
\begin{aligned}
 & (P1)& \underset{\left({\bf{r}}_k\right)_{k=0}^K,\left({\bf{u}}_k\right)_{k=0}^{K-1}}{\text{minimize}}
 & & & \sum\limits_{k=0}^{K-1}\| {\bf{u}}_k \|_{\ell_1}\\
 & &\text{subject to}
 & & & {{\bf{r}}_{k + 1}} = {{\bf{A}}}{{\bf{r}}_k} + {{\bf{B}}}{{\bf{u}}_k}\\
 & &
 & & & {{\bf{y}}_k} = {\bf{C}}{{\bf{r}}_k}.
 \end{aligned}
\end{equation}
In the case that either input disturbance, or measurement noise, or both exist, we solve the following convex optimization Problem ($P2$):
\begin{equation}\label{P2}
\begin{aligned}
 & (P2)& \underset{\left({\bf{r}}_k\right)_{k=0}^K,\left({\bf{u}}_k\right)_{k=0}^{K-1}}{\text{minimize}}
 & & & \sum\limits_{k=0}^{K-1}\| {\bf{u}}_k \|_{\ell_1}\\
 & &\text{subject to}
 & & & {{\bf{r}}_{k + 1}} = {{\bf{A}}}{{\bf{r}}_k} + {{\bf{B}}}{{\bf{u}}_k}\\
 & &
 & & & \|{{\bf{y}}_k} - {\bf{C}}{{\bf{r}}_k}\|_{\ell_2} \leq \epsilon^{\prime\prime},\\
 \end{aligned}
\end{equation}
where $\epsilon^{\prime\prime}$ is 2-norm of a surrogate parameter that aggregates the effects of disturbance and noise.  In the case of noisy measurement with no disturbance $\epsilon^{\prime\prime}=\epsilon$.  
In the next section, we show conditions for the network \eqref{linearDynamicFormulation} under which Problems ($P1$) and ($P2$) result in exact and stable solutions. 


\section{Results}
We will develop our results in several steps.  First, we consider two cases for the observation matrix ${\bf{C}}$ in the absence of input disturbance, and proceed to establish existence and performance guarantees for solutions to the convex problems ($P1$) and ($P2$) for each case. After that, we continue the analysis to characterize the ability of a network to reject input disturbances while still admitting stable recovery in the presence of disturbance and noise, simultaneously. 

\subsection{Preliminaries}
We begin by recalling some basic matrix notation and matrix norm properties that will be used throughout this paper. Given normed spaces $(\mathbb{R}{^{n_1}}, \|.\|_{\ell_2})$ and $(\mathbb{R}{^{n_2}}, \|.\|_{\ell_2})$, the corresponding induced norm or operator norm denoted by $\| .\|_{i,2}$ over linear maps ${\bf{D}}: \mathbb{R}{^{n_1}} \to \mathbb{R}{^{n_2}}$, ${\bf{D}}\in\mathbb{R}{^{n_2\times n_1}}$ is defined by
\begin{equation}
\begin{aligned}
\| {\bf{D}}\|_{i, 2} &= sup\{ \frac{\| {\bf{D}} {\bf{r}} \|_{\ell_2}} {\|{\bf{r}}\|_{\ell_2}} ~|~ {\bf{r}} \in \mathbb{R}{^{n_1}} \}\\
 & = max\{ \sqrt{\lambda} ~|~ \lambda \in \sigma\left({\bf{D}}^T{\bf{D}}\right)  \},
\end{aligned}
\label{induced}
\end{equation}
where $\sigma({\bf{M}})$ is the set of eigenvalues of ${\bf{M}}$ (or the spectrum of ${\bf{M}}$).

\begin{defn}
A vector is said to be $s$-sparse if $\|{\bf{c}}\|_{\ell_0} \leq s$, in other words it has at most $s$ nonzero entries.
\end{defn}
It is well known that in the static case (standard CS), exact and stable recovery of sparse inputs can be obtained under the restricted isometry property (RIP)  \cite{donoho2006most,candes2006robust,candes2005decoding,candes2006stable,candes2008restricted}, defined as:
\begin{defn}{ The restricted isometry constant $\delta_s$ of a matrix \mbox{$\Phi \in\mathbb{R}{^{n\times m}}$} is defined as the smallest number such that for all $s$-sparse vectors ${\bf{c}} \in\mathbb{R}{^{m}}$ the following equation holds
\begin{equation}
(1-\delta_{s})\|{\bf{c}}\|_{\ell_2}^{2}\leq\|\Phi {\bf{c}}\|_{\ell_2}^{2}\leq(1+\delta_{s})\|{\bf{c}}\|_{\ell_2}^{2}.
\label{RIP}
\end{equation}
 }
\end{defn}
It is known that many types of random matrices with independent and identically distributed entries or sub-Gaussian matrices satisfy the RIP condition (\ref{RIP}) with overwhelming probability  \cite{candes2006near,baraniuk2008simple,mendelson2008uniform}.

\subsection{{Case 1: Full-rank Square Observation Matrix $C$ without Input Disturbance}}

In the first case, we consider (\ref{linearDynamicFormulation}) in the absence of input disturbance (${\bf{d}}_k=0$) and we assume that the linear map \mbox{$C: \mathbb{R}{^{n}} \to \mathbb{R}{^{n}}$, $p=n$}, has no nullspace, $\mathcal{N}(C)=\{0\}$, which means that the network states can be exactly recovered by inverting the observation equation \eqref{linearDynamicFormulation} (the trivial case being C equal to the identity). Our first result establishes a one to one correspondence between sparse input and observed output for the system (\ref{linearDynamicFormulation}). 

\begin{lem}
Suppose that the sequence $\left({\bf{y}}_k\right)_{k=0}^K$ from noiseless measurements is given, and ${\bf{A}}$, ${\bf{B}}$, ${\bf{C}}$, $\mathcal{N}(C)=\{0\}$ are known. Assume the matrix ${\bf{B}}$ satisfies the RIP condition (\ref{RIP}) with isometry constant $\delta_{2s}<1$. Then, there is a unique $s$-sparse sequence of $\left({\bf{u}}_k\right)_{k=0}^{K-1}$ and a unique sequence of  $\left({\bf{r}}_k\right)_{k=0}^K$ that generate $\left({\bf{y}}_k\right)_{k=0}^K$. 
\label{UniquenessFullRankC} 
\end{lem}

\begin{proof}
See Appendix \ref{AppA}.
\end{proof}
Having established the existence of a unique solution, we now proceed to study convex Problems ($P1$) and ($P2$) that recover these solutions. First, we provide theoretical results for the stable recovery of the input where measurements are noisy i.e., ($P2$).

\begin{thm}{\textit{(Noisy recovery)}}      \label{TheoremNoisyFullRank}
Assume that the matrix \mbox{${\bf{B}}$} satisfies the RIP condition (\ref{RIP}) with $\delta_{2s} < \sqrt{2}-1$. Suppose that the sequence $\left({\bf{y}}_k\right)_{k=0}^K$ is given and generated from sequences $\left({\bf{\bar{r}}}_k\right)_{k=0}^K$ and $s$-sparse $\left({\bf{\bar{u}}}_k\right)_{k=0}^{K-1}$ based on 
\begin{equation}
\begin{aligned}
{\bf{\bar{r}}}_{k+1} &= {\bf{A}} {\bf{\bar{r}}}_k + {\bf{B}} {\bf{\bar{u}}}_k,~k=0,\cdots,K-1\\
{\bf{y}}_k &= {\bf{C}} {\bf{\bar{r}}}_k + {\bf{e}}_k,~k=0,...,K,
\end{aligned}
\label{linearDynamicFormulationwithnoise}
\end{equation}
where $\left(\| {\bf{e}}_k \|_{\ell_2} \leq \epsilon\right)_{k=0}^K$  and ${\bf{A}}$, ${\bf{B}}$, ${\bf{C}}$, $\mathcal{N}({\bf{C}})=\{0\}$ are known. Then, the solution to Problem ($P2$) obeys
\begin{equation}
\sum_{k=0}^{K-1} \| {\bf{u}}^*_k - {\bf{\bar{u}}}_k \|_{\ell_2} \leq C_s \epsilon,
\label{ErrorBound}
\end{equation}
where
\begin{equation}
C_s = 2\alpha C_0  K(1-\rho)^{-1}.
\end{equation}
$C_0,~ \rho,~ \alpha$ are given explicitly below:
\begin{equation}
\begin{aligned}
&C_0 = \frac{1}{\sqrt{\sigma}} \left( 1 + \sqrt{\frac{\sigma_{max}\left({\bf{{C}}}^{T}{\bf{{C}}}\right) \sigma_{max}\left({\bf{{A}}}^{T}{\bf{{A}}}\right)}{ \sigma_{min}\left({\bf{{C}}}^{T}{\bf{{C}}}\right)}} \right),\\
&\alpha = \frac{2 \sqrt{1+ \delta_{2s}}}{1-\delta_{2s}},\\
&\rho = \frac{\sqrt{2} \delta_{2s}}{1 - \delta_{2s}},\\
&{\sigma_{min}\left({\bf{{C}}}^{T}{\bf{{C}}}\right)} < \sigma < { \sigma_{max}\left({\bf{{C}}}^{T}{\bf{{C}}}\right)}.
\end{aligned}
\end{equation} 
\end{thm}
\begin{proof}{
Assume that the the sequences $\left({\bf{{r}}}^*_k\right)_{k=0}^K$ and sparse $\left({\bf{{u}}}^*_k\right)_{k=0}^{K-1}$ are the solutions of Problem ($P2$). First we derive the bound for the $\| {\bf{{r}}}^*_k  -  {\bf{{\bar{r}}}}_k \|_{\ell_2}$ in the following Lemma.

\begin{lem}
Suppose that the sequence $\left({\bf{y}}_k\right)_{k=0}^K$ is given and generated from sequences $\left({\bf{\bar{r}}}_k\right)_{k=0}^K$ and $s$-sparse $\left({\bf{\bar{u}}}_k\right)_{k=0}^{K-1}$ based on (\ref{linearDynamicFormulationwithnoise}), where $\left(\| {\bf{e}}_k \|_{\ell_2} \leq \epsilon\right)_{k=0}^K$  and ${\bf{A}}$, ${\bf{B}}$, ${\bf{C}}$, $\mathcal{N}({\bf{C}})=\{0\}$ are known. Then, any solution ${\bf{{r}}}^*_k$ to Problem ($P2$) obeys
\begin{equation}
\begin{aligned}
\| {\bf{{r}}}^*_k  -  {\bf{{\bar{r}}}}_k \|_{\ell_2} &\leq \frac{2 \epsilon}{ \sqrt{ \sigma_{min}\left({\bf{{C}}}^{T}{\bf{{C}}}\right)}}
\end{aligned}
\end{equation}
\label{lemma1}
\end{lem}

\begin{proof}
See Appendix \ref{AppB}
\end{proof}

From  Lemma \ref{lemma1}, non-singularity of ${\bf{C}}$ and the equation ${\bf{y}}_k = {\bf{CA}}{\bf{\bar{r}}}_{k-1} + {\bf{CB}}{\bf{\bar{u}}}_{k-1} + {\bf{e}}_k$ we can derive a bound for $\|{\bf{B}}\left( {\bf{u}}^*_k - {\bf{\bar{u}}}_k \right)\|_{\ell_2}$ as 
\begin{equation}
\begin{aligned}
&{\sqrt{\sigma}}\|{\bf{B}}\left( {\bf{u}}^*_k - {\bf{\bar{u}}}_k \right)\|_{\ell_2} =\|{\bf{CB}}\left( {\bf{u}}^*_k - {\bf{\bar{u}}}_k \right)\|_{\ell_2}\\
&\qquad\qquad  =\|({\bf{e}}^*_{k+1} + {\bf{e}}_{k+1}) + {\bf{CA}}({\bf{\bar{r}}}_k - {\bf{r}}^*_k) \|_{\ell_2}\\
&\qquad\qquad  \leq \| {\bf{e}}_{k+1} + {\bf{e}}^*_{k+1}  \|_{\ell_2} + \| {\bf{CA}}({\bf{\bar{r}}}_k - {\bf{r}}^*_k) \|_{\ell_2}\\
&\qquad\qquad \leq 2 \epsilon \left( 1 + \sqrt{\frac{\sigma_{max}\left({\bf{{C}}}^{T}{\bf{{C}}}\right) \sigma_{max}\left({\bf{{A}}}^{T}{\bf{{A}}}\right)}{ \sigma_{min}\left({\bf{{C}}}^{T}{\bf{{C}}}\right)}} \right),
\end{aligned}
\end{equation}
which results in
\begin{equation}
\|{\bf{B}}\left( {\bf{u}}^*_k - {\bf{\bar{u}}}_k \right)\|_{\ell_2} \leq 2   C_0 \epsilon.
\label{main}
\end{equation}

Now, denote \mbox{${\bf{{u}}}^*_k = {\bf{{\bar{u}}}}_k + {\bf{{h}}}_k$} where ${\bf{{h}}}_k$ can be decomposed into a sum of vectors ${\bf{{h}}}_{k,T_0(k)}, {\bf{{h}}}_{k,T_1(k)},{\bf{{h}}}_{k,T_2(k)},\cdots$ for each $k$, each of sparsity at most $s$. Here, $T_0(k)$ corresponds to the location of non-zero elements of ${\bf{{\bar{u}}}}_k$, $T_1(k)$ to the location of $s$ largest coefficients of ${\bf{{h}}}_{k,T_0^c(k)}$, $T_2(k)$ to the location of the next $s$ largest coefficients of ${\bf{{h}}}_{k,T_0^c(k)}$, and so on. Also, let $T_{0 1}(k) \equiv T_0(k) \cup T_1(k)$. Extending the technique in  \cite{candes2006stable,candes2008restricted}, it is possible to obtain a cone constraint for the input in the linear dynamical systems.
\begin{lem}{\textit{(Cone constraint)}}
The optimal solution for the input in Problem ($P2$) satisfies
\begin{equation}
\sum_{k=0}^{K-1} \|{\bf{{h}}}_{k,T_{01}^c(1)}\|_{\ell_2}   \leq  \sum_{k=0}^{K-1}\|{\bf{{{h}}}}_{k,T_0(1)}\|_{\ell_2}.
\label{coneeq}
\end{equation}
\label{cone}
\end{lem}
\begin{proof}
See Appendix \ref{AppCone}.
\end{proof}

We can further establish a bound for the right hand side of (\ref{coneeq}): 

\begin{lem}
The optimal solution for the input in Problem ($P2$) satisfies the following constraint
\begin{equation}
\sum_{k=0}^{K-1} \|{\bf{{h}}}_{k,T_{01}(1)}\|_{\ell_2}  \leq K(1-\rho)^{-1}\alpha C_0 \epsilon.
\end{equation}

\label{lemma3}
\end{lem}
\begin{proof}
See Appendix \ref{AppD}.
\end{proof}

Finally, based on  Lemma \ref{cone} and  Lemma \ref{lemma3}, it is easy to see that
\begin{equation}
\begin{aligned}
\sum_{k=0}^{K-1}\|{\bf{{h}}}_{k}\|_{\ell_2}   &   \leq  \sum_{k=0}^{K-1} \left(\|{\bf{{h}}}_{k,T_{01}(1)}\|_{\ell_2} +  \|{\bf{{h}}}_{k,T^c_{01}(1)}\|_{\ell_2} \right) \\
& \leq 2 \sum_{k=0}^{K-1} \|{\bf{{h}}}_{k,T_{01}(1)}\|_{\ell_2}\\
&\leq 2\alpha C_0 K(1-\rho)^{-1}\epsilon = C_s \epsilon.
\end{aligned}
\end{equation}}
\end{proof}


We now state a Theorem that characterizes the solution for the noiseless case (P1), which follows as a special case of (P2) as the noise variance approaches zero.

\begin{thm}{\textit{(Noiseless recovery)}}  \label{TheoremNoiselessFullRank}
Assume that the matrix ${\bf{B}}$ satisfies the RIP condition (\ref{RIP}) with $\delta_{2s} < \sqrt{2}-1$. Suppose that the sequence $\left({\bf{y}}_k\right)_{k=0}^K$ is given and generated from sequences $\left({\bf{\bar{r}}}_k\right)_{k=0}^K$ and $s$-sparse input $\left({\bf{\bar{u}}}_k\right)_{k=0}^{K-1}$ based on dynamical equation (\ref{linearDynamicFormulation}), and ${\bf{A}}$, ${\bf{B}}$, ${\bf{C}}$, $\mathcal{N}(C)=\{0\}$ are known. Then the sequences $\left({\bf{\bar{r}}}_k\right)_{k=0}^K$ and $\left({\bf{\bar{u}}}_k\right)_{k=0}^{K-1}$ are the unique minimizer to Problem ($P1$).
\end{thm}

\begin{proof}
It is sufficient to consider $\epsilon = 0$ in equation (\ref{proof1}) which results in \mbox{$\|{\bf{{h}}}_{0,T_{01}(1)}\|_{\ell_2} = \cdots = \|{\bf{{h}}}_{K-1,T_{01}(K-1)}\|_{\ell_2} = 0$} from equation (\ref{proof2}), which implies that all elements of vectors ${\bf{{h}}}_0,\cdots,{\bf{{h}}}_{K-1}$ are zero and $\left({\bf{{u}}}^*_k = {\bf{{\bar{u}}}}_k\right)_{k=0}^{K-1}$.
\end{proof}

\subsection{{Case 2: Observation Matrix $C$ Satisfying Observability Condition without Input Disturbance}}
In this case, we consider (\ref{linearDynamicFormulation}) in the absence of input disturbance (${\bf{d}}_k=0$) with the linear map $C: \mathbb{R}{^{n}} \to \mathbb{R}{^{p}}$, $p<n$.   Thus, direct inversion of $C$ is not possible in this case.  For any positive $K$, we define the standard linear observability matrix as

\begin{equation}
\mathcal{O}_K \equiv \begin{pmatrix}
{\bf{C}}\\
{\bf{C}} {\bf{A}}\\
\vdots\\
{\bf{C}}{\bf{A}}^K
\end{pmatrix}
.
\end{equation}
If $rank(\mathcal{O}_K)=n$, then the system (\ref{linearDynamicFormulation}) is observable in the classical sense \footnote{The system is said to be observable if, for any initial state and for any known sequence of input there is a positive integer $K$ such that the initial state can be recovered from the outputs ${\bf{y}}_0$, ${\bf{y}}_1$,..., ${\bf{y}}_K$.
}. However, we do not assume any knowledge of the input other than the fact that it is $s$-sparse at each time.  Note that if we simply iterate the output equation in (\ref{linearDynamicFormulation}) for $K+1$ time steps and exploit the fact that the input vector is $s$-sparse as shown in Fig. \ref{SchemObservability}, we obtain:
\begin{equation}
\begin{pmatrix}
{\bf{y}}_0\\
{\bf{y}}_1\\
\vdots \\
{\bf{y}}_K
\end{pmatrix}
= \mathcal{O}_K {\bf{r}}_0 + \mathcal{J}_K^s  
\begin{pmatrix}
{\bf{u}}_0^s\\
{\bf{u}}_1^s\\
\vdots \\
{\bf{u}}_{K-1}^s
\end{pmatrix}
,
\label{SeqObserv}
\end{equation}
where $\mathcal{J}_K^s$ is as follows:
\begin{equation}
\mathcal{J}_K^s = 
\begin{pmatrix}
{\bf{0}} & {\bf{0}} &  \cdots & {\bf{0}} \\
{\bf{C}} {\bf{B}}_0^s & {\bf{0}} & \cdots & {\bf{0}} \\
{\bf{C}} {\bf{A}} {\bf{B}}_0^s & {\bf{C}} {\bf{B}}_1^s  & \cdots & {\bf{0}} \\
\vdots & \vdots & \ddots & \vdots \\
{\bf{C}} {\bf{A}}^{K-1} {\bf{B}}_0^s & {\bf{C}} {\bf{A}}^{K-2} {\bf{B}}_1^s  & \cdots &  {\bf{C}} {\bf{B}}_{K-1}^s 
\end{pmatrix}
,
\label{Jks}
\end{equation}
where ${\bf{B}}_i^s$ is the $n \times s$ matrix corresponding the active columns of ${\bf{B}}$ (corresponding to nonzero input entries) at time step $i$ (see Fig. \ref{SchemObservability}). In general, we do not know where active columns of ${\bf{B}}$ are located at each time \textit{a priori}. We define ${\bf{J}}_K^s$ as the set of all possible matrices satisfying the structure in \eqref{Jks}, where the cardinality of this set is ${{m}\choose{s}}^K$.

\begin{figure}[!t]
\centering
\vspace{-.2in}
\includegraphics[width=0.65\textwidth]{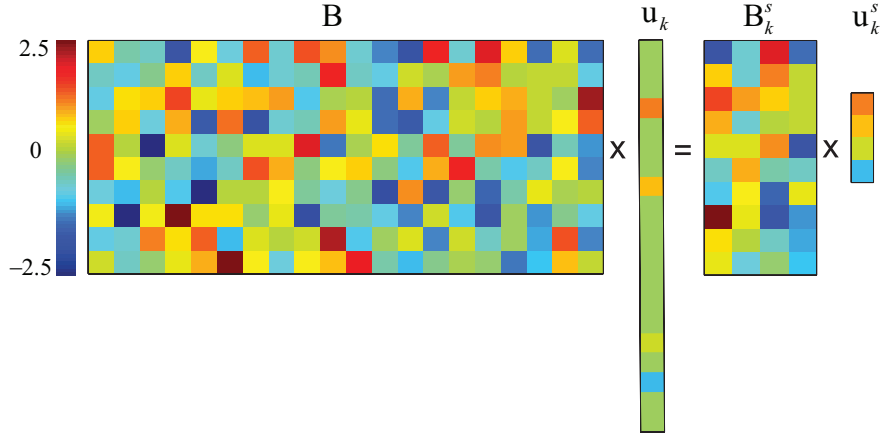}
\vspace{-.3in}
\caption{The matrix ${\bf{B}}_k^s$ is the $n \times s$ matrix corresponding the active columns of the full matrix ${\bf{B}}$ at time step $k$.} 
\label{SchemObservability}
\end{figure}

In the next Theorem, we establish conditions under which a one to one correspondence exists between sparse input and observed output for the system (\ref{linearDynamicFormulation}).

\begin{lem}
Suppose that the sequence $\left({\bf{y}}_k\right)_{k=0}^K$ from noiseless measurements is given, and ${\bf{A}}$, ${\bf{B}}$ and ${\bf{C}}$ are known. Assume $rank(\mathcal{O}_K)=n$ and the matrix ${\bf{CB}}$ satisfies the RIP condition (\ref{RIP}) with isometry constant $\delta_{2s}<1$. 
Further, assume 
\begin{equation}
rank([\mathcal{O}_K~~\mathcal{J}_K^{2s}])=n+rank(\mathcal{J}_K^{2s}),~ \forall \mathcal{J}_K^{2s} \in {\bf{J}}_K^{2s}.
\label{RankCondition}
\end{equation}
Then, there is a unique $s$-sparse sequence of $\left({\bf{u}}_k\right)_{k=0}^{K-1}$ and a unique sequence of $\left({\bf{r}}_k\right)_{k=0}^K$ that generate $\left({\bf{y}}_k\right)_{k=0}^K$. 
\label{UniquenessGeneralC} 
\end{lem}

\begin{proof}
See Appendix \ref{AppE}.
\end{proof}

{
\begin{rem}
The rank condition implies that all columns of the observability matrix must be linearly independent of each other (i.e., the network is observable in the classical sense) and of all columns of $\mathcal{J}_K^s$.  Since the exact location of the nonzero elements of the input vector are not known \textit{a priori}, this condition is specified over all $\mathcal{J}_K^{2s}$. Thus, \eqref{RankCondition} is a combinatorial condition.  From our simulation studies, we observe that this condition holds for random Gaussian matrices almost always and, moreover, can be numerically verified for certain salient random networks (see also Example 3 in Section \ref{examples}). 
\end{rem}
}

Having established the existence of a unique solution, we now proceed to study the convex problems ($P1$) and ($P2$) that recover these solutions for this case. First, we provide theoretical results for the stable recovery of the input where measurements are noisy i.e., (P2).

\begin{thm}{\textit{(Noisy recovery)}} 
Assume $rank(\mathcal{O}_K)=n$ and the matrix \mbox{${\bf{CB}}$} satisfies the RIP condition (\ref{RIP}) with $\delta_{2s} < \sqrt{2}-1$. assume \eqref{RankCondition} holds. Suppose that the sequence $\left({\bf{y}}_k\right)_{k=0}^K$ is given and generated from sequences $\left({\bf{\bar{r}}}_k\right)_{k=0}^K$ and $s$-sparse $\left({\bf{\bar{u}}}_k\right)_{k=0}^{K-1}$ based on \eqref{linearDynamicFormulationwithnoise}, where $\left(\| {\bf{e}}_k \|_{\ell_2} \leq \epsilon\right)_{k=0}^K$  and ${\bf{A}}$, ${\bf{B}}$ and ${\bf{C}}$ are known. Then, any $s$-sparse solution of Problem ($P2$) obeys
\begin{equation}
\exists~ C_s ~such~that~\sum_{k=0}^{K-1} \| {\bf{u}}^*_k - {\bf{\bar{u}}}_k \|_{\ell_2} \leq C_s \epsilon.
\label{ErrorBoundObservability}
\end{equation}
\label{TheoremNoisyObservability}
\end{thm}

\begin{proof}
See Appendix \ref{AppF}.
\end{proof}

\begin{thm}{\textit{(Noiseless recovery)}} 
Assume $rank(\mathcal{O}_K)=n$ and the matrix \mbox{${\bf{CB}}$} satisfies the RIP condition (\ref{RIP}) with \mbox{$\delta_{2s} < \sqrt{2}-1$}. Further, assume \eqref{RankCondition} holds.
Suppose that the sequence $\left({\bf{y}}_k\right)_{k=0}^K$ is given and generated from sequences $\left({\bf{\bar{r}}}_k\right)_{k=0}^K$ and $s$-sparse inputs $\left({\bf{\bar{u}}}_k\right)_{k=0}^{K-1}$ based on dynamical equation (\ref{linearDynamicFormulationwithnoise}), where $\epsilon=0$  and ${\bf{A}}$, ${\bf{B}}$ and ${\bf{C}}$ are known. Then the sequences $\left({\bf{\bar{r}}}_k\right)_{k=0}^K$ and $\left({\bf{\bar{u}}}_k\right)_{k=0}^{K-1}$ are the unique minimizer to Problem ($P1$).
\label{TheoremNoiselessObservability}
\end{thm}

\begin{proof}
It can be concluded from Lemma \ref{UniquenessGeneralC} and Theorem \ref{TheoremNoisyObservability} that with the assumption stated in the theorem and $\epsilon=0$ the sequences $\left({\bf{\bar{r}}}_k\right)_{k=0}^K$ and $\left({\bf{\bar{u}}}_k\right)_{k=0}^{K-1}$ are the unique minimizer to Problem ($P1$). 
\end{proof}

{
\begin{rem}
Imposing an RIP condition on the combined matrix ${\bf{CB}}$ bears some conceptual similarity to the formulation of an overcomplete dictionary in the classical compressed sensing literature \cite{candes2006fast,candes2004new}.  In this sense, the ${\bf{B}}$ matrix (i.e., the afferent stage) can be interpreted as a dictionary that transforms the sparse input ${\bf{u}}$ onto the recurrent network states.
\end{rem}
}
\subsection{{{Case 3: Optimal Network Design to Enable Recovery in the Presence of Disturbance and Noise}}}\label{secDisNoise}
Finally, we show how eigenstructure of the network implies a fundamental tradeoff between stable recovery and rejection of disturbance (i.e., corruption).

It is easy to see from \eqref{ErrorBound} that the upper-bound of the recovery error is reduced by decreasing the maximum singular value of ${\bf{A}}$.  Thus, from now on we use the upper-bound of the input recovery error as a comparative measure of performance. In the absence of both disturbance and noise, the best error performance is achieved when ${\bf{A}=0}$, i.e., the network is static, which in intuitive since in this scenario any temporal effects would smear the salient parts of the signal. 

On the other hand, having dynamics in the network should improve the error performance in the presence of the disturbance.  To demonstrate this, consider \eqref{linearDynamicFormulation} with $\bf{d}_k$ nonzero. When ${\bf{A}=0}$, i.e., a static network, the disturbance can be exactly transformed to the measurement equation resulting in $ {\bf{C}} {\bf{d}}_k+{\bf{e}}_k$ as a surrogate measurement noise with
\begin{equation}
\begin{aligned}
&{\|\bf{C}} {\bf{d}}_k+{\bf{e}}_k\|_{\ell_2} \leq \sqrt{\sigma^{\prime}}\epsilon^{\prime} + \epsilon,\\
&{\sigma_{min}\left({\bf{{C}}}^{T}{\bf{{C}}}\right)} < \sigma^{\prime} < { \sigma_{max}\left({\bf{{C}}}^{T}{\bf{{C}}}\right)}.
\end{aligned}
\end{equation}
In this case, the error upper-bound can be obtained by exploiting the result of Theorem \ref{TheoremNoisyFullRank} as
\begin{equation}
\begin{aligned}
&\sum_{k=0}^{K-1} \| {\bf{u}}^*_k - {\bf{\bar{u}}}_k \|_{\ell_2} \leq C^{\prime}_s (\sqrt{\sigma^{\prime}}\epsilon^{\prime} + \epsilon),\\
&C^{\prime}_s = \frac{2}{\sqrt{\sigma}}\alpha   K(1-\rho)^{-1}.
\end{aligned}
\label{BoundWithoutA}
\end{equation}

When ${\bf{A}}$ is nonzero, it is not possible to exactly map the disturbance to the output as above.  Nevertheless, it is straightforward to approximate the relative improvement in performance.  For instance, consider a system with ${\bf{A}}$ symmetric and where the disturbance and input are in displaced frequency bands.  Then it is a direct consequence of linear filtering that the power spectral density of the disturbance can be attenuated according to  
\begin{equation}
{\bf{\mathcal{S}}}_{{\bf{{d}}^{Filt}}}(e^{j\omega})={s_{\bf{d}}(e^{j\omega})}(e^{j\omega}{\bf{I}}_n - {\bf{A}})^{-1}  (e^{-jw}{\bf{I}}_n - {\bf{A}})^{-1},
\end{equation}
where $\omega$ is the frequency of the disturbance.  So, for instance, if $\omega=\pi$,
\begin{equation}
\begin{aligned}
\Tr \{ {\bf{\mathcal{S}}}_{{\bf{{d}}^{Filt}}}(e^{j\pi}) \} &= \Tr \{ {s_{\bf{d}}(e^{j\pi})} ({\bf{I}}_n + {\bf{A}})^{-2} \} \\
& = {s_{\bf{d}}(e^{j\pi})}\sum_{i=1}^n (1 + \lambda_i (A) )^{-2},
\end{aligned}
\label{powerSpectrum}
\end{equation}
where $\lambda_i (A)$ is the $i^{th}$ eigenvalue of matrix ${\bf{A}}$.  
Assuming the input is sufficiently displaced in frequency from the disturbance, the error upper-bound can be then readily approximated using the results of Theorem \ref{TheoremNoisyFullRank} as follows
\begin{equation}
\begin{aligned}
& \sum_{k=0}^{K-1} \| {\bf{u}}^*_k - {\bf{\bar{u}}}_k \|_{\ell_2} \leq C_s \left(\frac{{\sqrt{\sigma^{\prime\prime}}}}{n}{\sum_{i=1}^n (1 + \lambda_i (A) )^{-2}} \epsilon^{\prime} + \epsilon \right), \\
&\qquad\qquad\quad{\sigma_{min}\left({\bf{{C}}}^{T}{\bf{{C}}}\right)} < \sigma^{\prime\prime} < { \sigma_{max}\left({\bf{{C}}}^{T}{\bf{{C}}}\right)}.
\end{aligned}
\label{BoundWithA}
\end{equation}
By comparing \eqref{BoundWithoutA} and \eqref{BoundWithA}, it is easy to verify that {\bf{A}} and {\bf{C}} can be designed in a way to reduce error upper-bound at least by a factor of two, assuming $\sigma^{\prime}$ and $\sigma^{\prime\prime}$ are close to each other.  In the examples below, we will show that, in fact, performance in many cases can exceed this bound considerably.

\section{Examples}\label{examples}
In this section, we present several examples that demonstrate the developed results. For solving our convex optimization problems, we used CVX with MATLAB interface \cite{cvx,gb08}. To create example networks, we generated the matrices ${\bf{A}}$, ${\bf{B}}$, ${\bf{C}}$ using a Gaussian random number generator in MATLAB.


\begin{figure}[!t]
\centering
\includegraphics[width=0.65\textwidth]{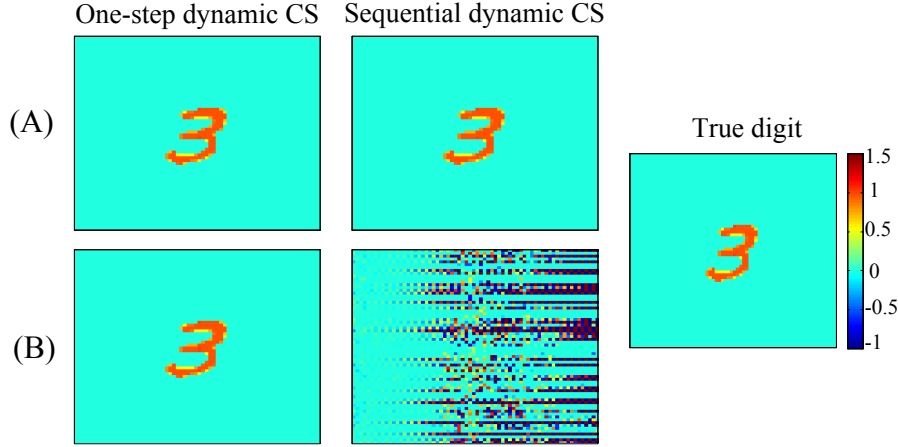}
\caption{The recovered input for (A) $p=n$ (B) $p=35$ for both static (middle images) and dynamic (left images) CS where $n=45$ and $m=68$. Original input is in the right hand side denoted as true digit.}
\label{ObservabilityCS}
\end{figure}

\subsection{Example 1: One-step and Sequential Recovery}
In this experiment, we consider a dynamical system with sparse input which satisfies conditions in Theorem \ref{TheoremNoisyObservability}. Here, we consider random Gaussian matrices for ${\bf{A}}$, ${\bf{B}}$ and ${\bf{C}}$, with $n=45$, $m=68$.  The input is defined as the image of a digit, shown in Fig. \ref{ObservabilityCS} with values between 0 and 1, where the horizontal axis is treated as time, i.e., column $k$ of the image is the input to the system at time $k$.

We proceed to perform input recovery in two ways: (i) by solving ($P1$) in one step over the entire horizon $K$, i.e., one-step recovery; and (ii) by solving ($P1$) $K$ times, sequentially, i.e., recovery at each time step.  We compare the outcomes for two cases:   

\paragraph{Full Rank $C$}  Fig. \ref{ObservabilityCS}A shows the recovered input for the case that $p=n$ for both one-step and sequential recovery, and it can be seen that sparse input can be recovered in two cases perfectly. This is expected, since in this case, $C$ can be inverted at each time step.  

\paragraph{$C$ Satisfying Observability Condition}  Fig. \ref{ObservabilityCS}B of the figure illustrates the results for the case that $p=35$, but where $C$ satisfies the observability condition.  It is clear that sequential dynamic CS can not recover the input exactly.  However, from our results (Theorem \ref{TheoremNoiselessObservability}) we expect that one-step recovery (over the entire horizon) is possible, as is evidenced in the figure.




\subsection{Example 2: Recovery in the Presence of Disturbance and Noise}

Fig. \ref{MSE}A shows the mean square error ($M\!S\!E$) versus the maximum singular value of ${\bf{A}}$, for several random realization of ${\bf{A}}$, in the case of full rank $C$. In this study, ${\bf{e}}_k$ is assumed to follow an uniform distribution $\mathcal{U}(-0.5, 0.5)$ while $\bf{d}_k=0$. It can be seen from this figure that by increasing $\sqrt{\sigma_{max}{({\bf{A}}^T {\bf{A}})}}$, the recovery performance is degraded, as we expect based on the derived bound for the error in \eqref{ErrorBound}.

\begin{figure}[!t]
	\centering
	\vspace{-.1in}
	\includegraphics[width=0.65\textwidth]{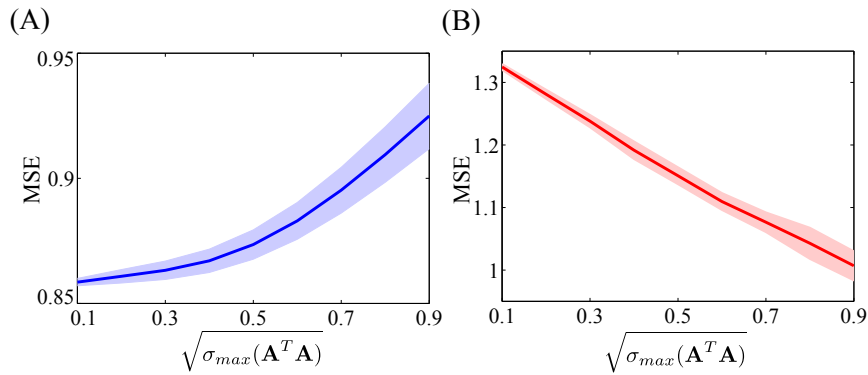}
	\caption{(A) $M\!S\!E$ versus the maximum singular value of ${\bf{A}}$, for several random realization of ${\bf{A}}$ with noise and in the absence of disturbance. (B) $M\!S\!E$ versus the maximum singular value of ${\bf{A}}$, for several random diagonal ${\bf{A}}$ with noise and disturbance.}
	\label{MSE}
\end{figure}

To contrast Fig. \ref{MSE}A, we consider the case when disturbance is added to the input. In Fig. \ref{MSE}B, we show the $M\!S\!E$ versus $\sqrt{\sigma_{max}{({\bf{A}}^T {\bf{A}})}}$ for several random diagonal matrices ${\bf{A}}$ when ${\bf{e}}_k \sim \mathcal{U}(-0.5, 0.5)$ and ${\bf{d}}_k \sim \mathcal{N}(0, 1)$. As expected from our results, $\sqrt{\sigma_{max}{({\bf{A}}^T {\bf{A}})}}$ can not be arbitrary small, since in this case the disturbance would entirely corrupt the input. 

{We conducted simulation experiments to examine the effect of the noise and disturbance strength on the reconstruction error. Fig. \ref{MSEvsNoisePower} shows the average $M\!S\!E$ for the reconstructed input versus $log(1/\epsilon)$ and $log(1/\epsilon^\prime)$, respectively with $n = 50$, $m = 100$ for $100$ random trials (different random matrices A, B, C in each trial). This figure shows that the reconstruction error decreases as a function of noise energy.}

\begin{figure}[!t] 
	\centering
	\includegraphics[width=0.65\textwidth]{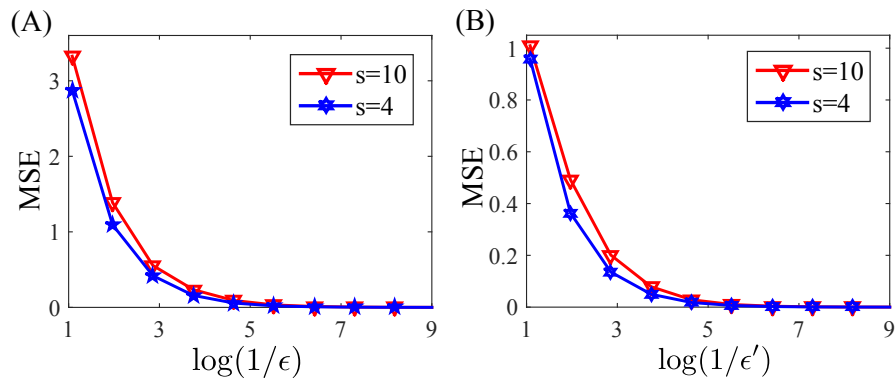}
	\caption{{$M\!S\!E$ as a function of (A) $log(1/\epsilon)$ and (B) $log(1/\epsilon^\prime)$ for the reconstructed input with $n = 50$, $m = 100$ over $100$ random trials (different random matrices $\textbf{A}$, $\textbf{B}$, $\textbf{C}$ in each trial).}}
	\label{MSEvsNoisePower}
	\vspace{.15 in}
\end{figure}

\begin{figure}[!t]
	\centering
	\vspace{-.1in}
	\includegraphics[width=0.7\textwidth]{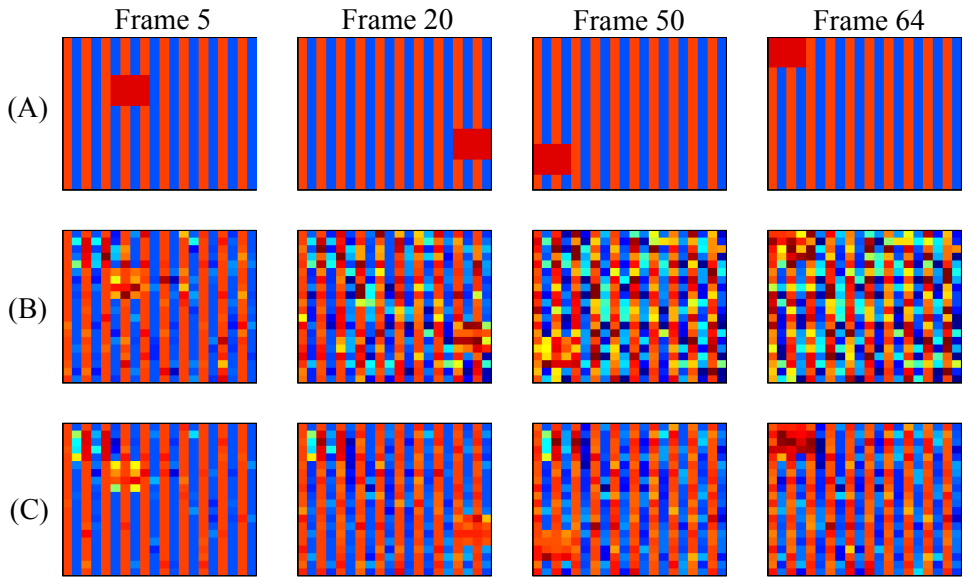}
	\caption{(A) Four noiseless frames of a movie. Recovery via  (B) static CS and (C) dynamic CS in the presence of disturbance.}
	\label{MovieSchemCSvsDCS}
\end{figure}

The next study illustrates recovery in the presence of both disturbance and noise for a smoothly changing sequence of 64 images (frames). Each frame, is corrupted with disturbance and at each time, and the difference between two consecutive frames is considered as the sparse input to the network. The disturbance ${\bf{d}}_k$ is assumed to be a random variable drown from a Gaussian distribution, $\mathcal{N}(0, 0.2)$, passed through a fifth-order Chebyshev high pass filter. Each frame has $m=400$ pixels and $K=64$. Furthermore, we consider random Gaussian matrices for ${\bf{B}}$ and ${\bf{C}}$ with $n=p=200$. 

We proceeded to design the matrix ${\bf{A}}$ to balance the performance bound \eqref{ErrorBound} and the ability to reject the disturbance as per Section \ref{secDisNoise}. Fig. \ref{MovieSchemCSvsDCS}A shows the original frames at different times. We assumed the first frame is known exactly. Frames recovered from the output of a static network , i.e., $\bf{A}=0$ are depicted in Fig. \ref{MovieSchemCSvsDCS}B.  In contrast, frames recovered from the output of the designed dynamic network are shown in Fig. \ref{MovieSchemCSvsDCS}C. It is clear from the figure that quality of recovery is better in the latter case. Fig. \ref{PSNRCSvsDCS} illustrates the PSNR, defined as $10\log(\frac{1}{M\!S\!E})$ as a function of frame number with and without dynamics. It can be concluded from this figure that having a designed matrix ${\bf{A}}$ results in recovery that is more robust to disturbance, while without dynamics, error propagates over time, and the reconstruction quality is degraded.

\begin{figure}[!t]
	\centering
	\vspace{-.1in}
	\includegraphics[width=0.44\textwidth]{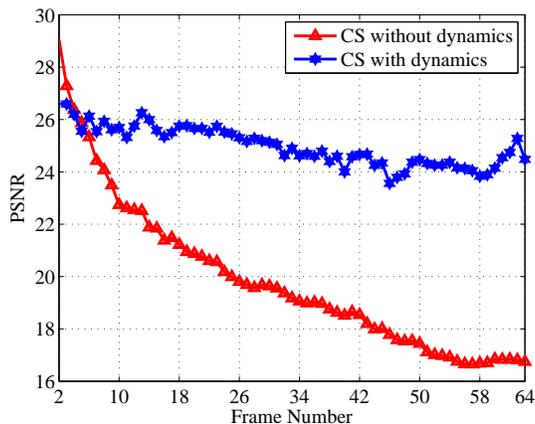}
	\caption{PSNR of the recovered frame versus frame number for both CS with and without dynamics.}
	\label{PSNRCSvsDCS}
\end{figure}

\subsection{Example 3: Input Recovery in an Overactuated Rate-Based Neuronal Network}

A fundamental question in theoretical neuroscience centers on how the architecture of brain networks enables the encoding/decoding of sensory information \cite{barth2012experimental,olshausen1996emergence}.  In our final example, we use the results of Theorems \ref{TheoremNoisyFullRank} and \ref{TheoremNoiselessObservability} to highlight how certain structural and dynamical features of neuronal networks may provide the substrate for sparse input decoding.  

Specifically, we consider a firing rate-based neuronal network \cite{dayan2005theoretical} of the form
\begin{equation}
	{\bf{T}_r} \frac{d{\bf{r}}}{dt}= -{\bf{r+ Wu + Mr}},
	\label{firing}
\end{equation}
with input rates ${\bf{u}} \in \mathbb{R}^m$, output rates ${\bf{r}} \in \mathbb{R}^n$, a feed-forward synaptic weight matrix ${\bf{W}}  \in \mathbb{R}^{n \times m}$, and a recurrent synaptic weight matrix \mbox{${\bf{M}} \in \mathbb{R}^{n \times n}$}.  We consider $n = 50$ neurons which receive synaptic inputs from $m=100$ afferent neurons, i.e., neurons that impinge on the network in question.  Here, ${\bf{T}_r} \in \mathbb{S}_{+}^n$ is a diagonal matrix whose diagonal elements are the time constants of the neurons.  A discrete version of (\ref{firing}), alongside a linear measurement equation can be written in the standard form \eqref{linearDynamicFormulation} where ${\bf{A}} = {\bf{I}}_n - \Delta t {\bf{T}_r}^{-1} + \Delta t {\bf{T}_r}^{-1} {\bf{M}}$ is related to connections between nodes in the network, and ${\bf{B}} = \Delta t {\bf{T}_r}^{-1} {\bf{W}}$ contains weights between input and output nodes.  For this example, we assume that the network connectivity has a Watts--Strogatz small-world topology  \cite{watts1998collective} with connection probability $p_M$ and rewiring probability $q_M$.

The recurrent synaptic matrix ${\bf{M}}$ is defined as 
\begin{equation}
	\begin{aligned}
		\left({\bf{M}}\right)_{ij} &= \left\{ 
		\begin{array}{l l}
			+ m^E_{ij} & ~ \text{if recurrent neuron $j$ is excitatory}\\
			0&  ~ \text{if no connection from neuron $j$ to $i$}\\
			- m^I_{ij}&  ~ \text{if recurrent neuron $j$ is inhibitory}
			
		\end{array} \right.\\
	\end{aligned} 
\end{equation}

\begin{figure}[!t]
	\centering
	\includegraphics[width=0.65\textwidth]{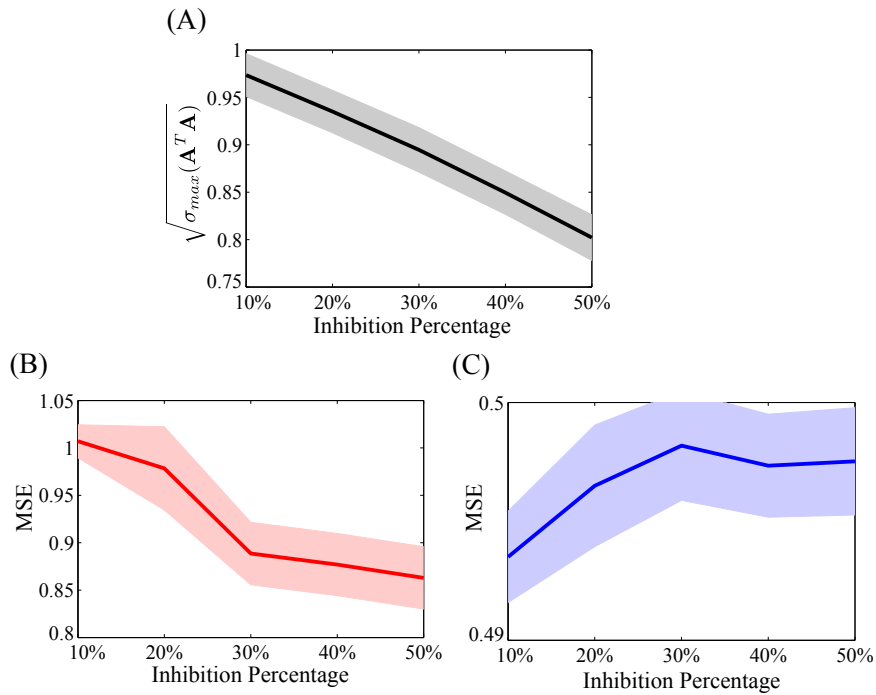}
	\caption{(A) The maximum singular value of ${\bf{A}}$ versus the inhibition percentage. $M\!S\!E$ of input recovery in the presence of (B) Noise and (C) Disturbance as a function of the percent of inhibitory neurons.}
	\label{WatsStrogatz}
\end{figure}


For the purposes of illustration, we select the diagonal elements of matrix ${\bf{T}}_r$, from a uniform distribution $\mathcal{U}(0.1,0.2)$.  We study the recovery performance associated with the network over 100 time steps, assuming a timescale of milliseconds and an discretization step of  $0.1~ms$.  At each time step, the nonzero elements of the input vector ${\bf{u}}$, i.e., firing rate of the afferent neurons, are drawn from an uniform distribution $\mathcal{U}(0.5, 1.5)$.  Moreover, we assume that elements of the observation matrix ${\bf{C}}$ are drawn from a Gaussian distribution $\mathcal{N}(0,1)$.  Finally, we assume $m^E_{ij}$ and $m^I_{ij}$ are drawn from lognormal distributions $\ln \mathcal{N}(0,1)$ and $\ln \mathcal{N}(0,0.1)$, respectively. The latter assumption is chosen for illustration only and is not related to known physiology.



\subsubsection{Recovery Performance from Error Bounds}

We proceed to conduct a Monte Carlo simulation of 100 different realizations of ${\bf{W}}$, ${\bf{M}}$ and ${\bf{C}}$.  
Fig. \ref{WatsStrogatz}A illustrates that the maximum singular value of the matrix ${\bf{A}}$ decreases as a function of the percent of inhibitory neurons.  Thus, we anticipate from our derived performance bounds that performance in terms of mean square error (MSE) should be best for networks with high inhibition in the presence of noise. This prediction bears out in Fig. \ref{WatsStrogatz}B, where we indeed observe a monotone relationship between MSE and inhibition. On the one hand, low-inhibition is favorable for facilitating recovery in the presence of disturbance depicted in Fig. \ref{WatsStrogatz}C. Such tradeoffs are interesting to contemplate when considering the functional advantages of network architectures observed in biology, such as the pervasive 80-20 ratio of excitatory to inhibitory neurons \cite{dayan2005theoretical, king2013inhibitory}. Together, Figs.  \ref{WatsStrogatz}B and \ref{WatsStrogatz}C illustrate how the excitatory-inhibitory ratio mediate a basic tradeoff in the capabilities of a rate-based neuronal network.

{

\subsubsection{Recoverable Sparsity based on Theorem \ref{TheoremNoiselessObservability}}
Having ascertained the performance tradeoff curves, we sought to characterize in more detail the level of recoverable sparsity with specific connection to Theorem \ref{TheoremNoiselessObservability} and \eqref{RankCondition}.  We considered networks as above, but with $p=30$ and 20/80 for the ratio of inhibitory/excitatory over 10 time steps for 100 random trials.  Thus, the output of the network is of lower dimension than the network state space and the observability matrix is of nontrivial construction.
Fig. \ref{RankTest}A shows that for this setup, the rank condition \eqref{RankCondition} holds up to $2s=27$.  Thus, Theorem \ref{TheoremNoiselessObservability} predicts that recover will be possible (to within the RIP condition on ${\bf{CB}}$) for signals with 13 nonzero elements.  Fig. \ref{RankTest}B validates this theoretical prediction by illustrating recovery performance in the absence of disturbance and noise for different values  $s$. It is observed that when the rank condition holds, reconstruction is perfect and that the probability of exact recovery is decreased by increasing $s$, as expected. 

}




\begin{figure}[!t]
	\centering
	\includegraphics[width=0.65\textwidth]{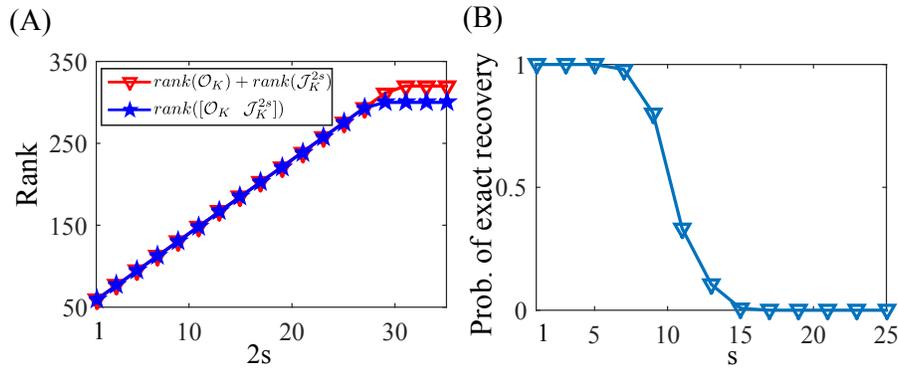}
	\caption{{(A) Examining the rank condition \eqref{RankCondition} in the networks of Example 3, with $p=30$, for different values of $s$. (B) The probability of exact recovery of dynamic sparse input to the network over $s$.}}
	\label{RankTest}
\end{figure}

{

\section{Conclusion}

\subsection{Summary}
In this paper, we present several results pertaining to the effect of temporal dynamics on compressed sensing of time-varying signals.  Specifically, we considered the recovery of sparse inputs to a linear dynamical system (network) from limited observation of the network states.  We provide basic conditions on the system that ensure solution existence and, further, derive several bounds in terms of the system dynamics for recovery performance in the presence of both input disturbance and observation noise.  We show that dynamics can play opposing roles in mediating accurate recovery with respect to these two different sources of corruption.  Thus, our results indicate tradeoffs that may inform the design of dynamical systems for time-varying compressed sensing.  These tradeoffs are illustrated through a series of examples, including one that highlights how the developed theory could be used to interrogate the functional role of inhibition in a neuronal network. 

\subsection{Implications and Future Work}
The results can have both engineering and scientific impacts.  In the former case, the goal may be to \textit{design} networks to process time-varying signals that are naturally sparse, such as high-dimensional neural data, or to be resilient to time-varying sparse perturbations.  In the latter case, the goal is to understand how the naturally occurring architectures of networks, such as those in the brain, confer advantages for processing of afferent signals.  In both cases, a precursor to further study are a set of verifiable conditions that overtly link network characteristics/dynamics to sparse input processing.  Our paper provides such conditions for networks with linear dynamics and develops illustrative examples that highlight these potential applications.  Treatment of systems with nonlinear dynamics, as well as a more detailed examination of random networks using the theory, are left as subjects for future work.

}



 
\appendices

\section{Proof of Lemma \ref{UniquenessFullRankC}} \label{AppA}
Based on the assumption on the null space of the linear map $C: \mathbb{R}{^{n}} \to \mathbb{R}{^{n}}$, given $\left({\bf{y}}_k\right)_{k=0}^K$, there is a unique sequence of $\left({\bf{r}}_k\right)_{k=0}^K$. We now prove the uniqueness of $\left({\bf{u}}_k\right)_{k=0}^{K-1}$. First, consider the following equations:
\begin{equation}
\begin{aligned}
{\bf{y}}_0 &= {\bf{C}} {\bf{r}}_0\\
{\bf{y}}_1 &= {\bf{C}} {\bf{A}} {\bf{r}}_0 + {\bf{C}} {\bf{B}} {\bf{u}}_0\\
&\vdots\\
{\bf{y}}_K &= {\bf{C}} {\bf{A}} {\bf{r}}_{K-1} + {\bf{C}} {\bf{B}} {\bf{u}}_{K-1}.\\
\end{aligned}
\label{uniqueness}
\end{equation}
The remainder of the proof is by contradiction. Let us assume that the sequence $\left({\bf{u}}_k\right)_{k=0}^{K-1}$ is not unique and there is another sequence of $s$-sparse $\left({\bf{\hat{u}}}_k\right)_{k=0}^{K-1}$ which satisfies (\ref{uniqueness}), leading to
\begin{equation}
\begin{aligned}
{\bf{y}}_0 &= {\bf{C}} {\bf{r}}_0\\
{\bf{y}}_1 &= {\bf{C}} {\bf{A}} {\bf{r}}_0 + {\bf{C}} {\bf{B}} {\bf{\hat{u}}}_0\\
&\vdots\\
{\bf{y}}_K &= {\bf{C}} {\bf{A}} {\bf{r}}_{K-1} + {\bf{C}} {\bf{B}} {\bf{\hat{u}}}_{K-1}.\\
\end{aligned}
\label{contradiction}
\end{equation}
Therefore, based on (\ref{uniqueness}) and (\ref{contradiction}) we can conclude that 
\begin{equation}
{\bf{C}} {\bf{B}} ({\bf{u}}_0 - {\bf{\hat{u}}}_0 ) = \cdots = {\bf{C}} {\bf{B}} ({\bf{u}}_{K-1} - {\bf{\hat{u}}}_{K-1} ) = {\bf{0}}.
\label{uniqness}
\end{equation}
Matrix ${\bf{C}}$ is non-singular, hence equation  (\ref{uniqness}) can be simplified as follows:
\begin{equation}
{\bf{B}} ({\bf{u}}_0 - {\bf{\hat{u}}}_0 ) = \cdots = {\bf{B}} ({\bf{u}}_{K-1} - {\bf{\hat{u}}}_{K-1} ) = {\bf{0}}.
\label{uniqness}
\end{equation}

Based on the assumption that the matrix ${\bf{B}}$ satisfies the RIP condition (\ref{RIP}) with isometry constant $\delta_{2s}<1$ and the fact that the support of the vectors $\left({\bf{u}}_0 - {\bf{\hat{u}}}_0\right),\cdots,~\left({\bf{u}}_{K-1} - {\bf{\hat{u}}}_{K-1}\right)$ are at most $2s$, the lower bound of the RIP condition for ${\bf{B}}$ results in
\begin{equation}
\begin{aligned}
(1-\delta_{2s})\|{\bf{u}}_0 - {\bf{\hat{u}}}_0\|_{\ell_2}^{2} &\leq \|{\bf{B}} ({\bf{u}}_0 - {\bf{\hat{u}}}_0)\|_{\ell_2}^{2} = 0\\
&\vdots\\
(1-\delta_{2s})\|{\bf{u}}_{K-1} - {\bf{\hat{u}}}_{K-1}\|_{\ell_2}^{2} &\leq \|{\bf{B}} ({\bf{u}}_{K-1} - {\bf{\hat{u}}}_{K-1})\|_{\ell_2}^{2} = 0,
\end{aligned}
\end{equation}
which means that ${\bf{u}}_0 = {\bf{\hat{u}}}_0,~\cdots,~{\bf{u}}_{K-1} = {\bf{\hat{u}}}_{K-1}$ and the sequence of $s$-sparse vectors $\left({\bf{u}}_k\right)_{k=0}^{K-1}$ is unique. 

\section{Proof of  {{Lemma}} \ref{lemma1}} \label{AppB}
If ${\bf{{r}}}^*_k$ is the solution to Problem ($P2$), then \mbox{${\bf{{y}}}^*_k= {\bf{{C}}} {\bf{{r}}}^*_k$} satisfies the inequality in ($P2$) which means that \mbox{$\| {\bf{y}}_k- {\bf{y}}^*_k \|_{\ell_2} \leq \epsilon$} which can be reformulated as
\begin{equation}
{\bf{y}}^*_k  =  {\bf{y}}_k + {\bf{e}}^*_k, ~ \|{\bf{e}}^*_k \|_{\ell_2} \leq \epsilon.
\label{key}
\end{equation}
By replacing ${\bf{y}}_k$ from (\ref{linearDynamicFormulationwithnoise}) in (\ref{key}) we have 
\begin{equation}
\begin{aligned}
{\bf{{C}}} {\bf{{r}}}^*_k  =  {\bf{{C}}} {\bf{{\bar{r}}}}_k + {\bf{e}}_k + {\bf{e}}^*_k\\
{\bf{{C}}} \left( {\bf{{r}}}^*_k  -  {\bf{{\bar{r}}}}_k \right) = {\bf{e}}_k + {\bf{e}}^*_k,
\end{aligned}
\end{equation}
which results in
\begin{equation}
\begin{aligned}
\| {\bf{{r}}}^*_k  -  {\bf{{\bar{r}}}}_k \|_{\ell_2} &= \| {\bf{{C}}}^{-1} \left(  {\bf{e}}_k + {\bf{e}}^*_k \right) \|_{\ell_2}.
\end{aligned}
\label{bound}
\end{equation}

Finally, we can derive the error bound for the state error at each time by substituting (\ref{induced}) into (\ref{bound}) as
\begin{equation}
\begin{aligned}
\| {\bf{{r}}}^*_k  -  {\bf{{\bar{r}}}}_k \|_{\ell_2} &\leq \sqrt{ \sigma_{max}\left({\bf{{C}}}^{-T}{\bf{{C}}}^{-1}\right)} ~\|  {\bf{e}}_k + {\bf{e}}^*_k  \|_{\ell_2}\\
&\leq \sqrt{ \sigma_{max}\left({\bf{{C}}}^{-T}{\bf{{C}}}^{-1}\right)} ~\left(\|  {\bf{e}}_k \|_{\ell_2}+ \|{\bf{e}}^*_k \|_{\ell_2} \right)\\
& = \frac{2 \epsilon}{ \sqrt{ \sigma_{min}\left({\bf{{C}}}^{T}{\bf{{C}}}\right)}}
\end{aligned}
\label{finalbound}
\end{equation}

\section{Proof of  {{Lemma}} \ref{cone}} \label{AppCone}
For each $j\geq 2$ and $k=0,\cdots,K-1$ we have
\begin{equation}
\|{\bf{{h}}}_{k,T_j(k)}\|_{\ell_2} \leq s^{1/2} \|{\bf{{h}}}_{k,T_j(k)}\|_{l_{\infty}} \leq s^{-1/2} \|{\bf{{h}}}_{k,T_{j-1}(k)}\|_{\ell_1},
\end{equation}
and thus
\begin{equation}
\begin{aligned}
\sum_{j\geq 2} \|{\bf{{h}}}_{k,T_j(k)}\|_{\ell_2} & \leq s^{-1/2} (\|{\bf{{h}}}_{k,T_1(k)}\|_{\ell_1} + \|{\bf{{h}}}_{k,T_2(k)}\|_{\ell_1} + \cdots) \\
 & \leq s^{-1/2} \|{\bf{{h}}}_{k,T_0^c(k)}\|_{\ell_1}.
\end{aligned}
\label{othersets}
\end{equation}
Therefore, we have the following equation
\begin{equation}
\begin{aligned}
\|{\bf{{h}}}_{k,T^c_{0 1}(k)}\|_{\ell_2}&=\|\sum_{j\geq 2} {\bf{{h}}}_{k,T_j(k)}\|_{\ell_2} \leq \sum_{j\geq 2} \|{\bf{{h}}}_{k,T_j(k)}\|_{\ell_2}\\
&  \leq s^{-1/2} \|{\bf{{h}}}_{k,T_0^c(k)}\|_{\ell_1}.
\label{initial}
\end{aligned}
\end{equation}
Since $\left({\bf{{u^*}}}_k\right)_{k=0}^{K-1}$ minimizes the cost function in Problem ($P2$),
\begin{equation}
{\small{
\begin{aligned}
 \sum_{k=0}^{K-1}  \|{\bf{{\bar{u}}}}_k\|_{\ell_1}  & \geq  \sum_{k=0}^{K-1}  \|{\bf{{u}}}^*_k\|_{\ell_1}  = \sum_{k=0}^{K-1} \|{\bf{{\bar{u}}}}_k + {\bf{{h}}}_{k}\|_{\ell_1} \\
& = \sum_{k=0}^{K-1} \left(   \sum_{i \in T_0(k)} | {\bf{{\bar{u}}}}_{k,i} + {\bf{{h}}}_{k,i}| +  \sum_{i \in T_0^c(k)} | {\bf{{\bar{u}}}}_{k,i} + {\bf{{h}}}_{k,i}|  \right) \\
& \geq   \sum_{k=0}^{K-1}  (  \|{\bf{{{\bar{u}}}}}_{k,T_0(k)}\|_{\ell_1} - \|{\bf{{{h}}}}_{k,T_0(k)}\|_{\ell_1} 
+ \|{\bf{{{h}}}}_{k,T^c_0(k)}\|_{\ell_1} \\
& \qquad ~~~ +\|{\bf{{{\bar{u}}}}}_{k,T^c_0(k)}\|_{\ell_1} )
\end{aligned}
}}
\end{equation}
${\bf{{{\bar{u}}}}}_{0},\cdots,{\bf{{{\bar{u}}}}}_{K-1}$ are non-zero for $T_0(0),\cdots,T_0(K-1)$, respectively. Therefore, 
\begin{equation}
\|{\bf{{{\bar{u}}}}}_{0,T^c_0(0)}\|_{\ell_1} =\cdots= \|{\bf{{{\bar{u}}}}}_{K-1,T^c_0(K-1)}\|_{\ell_1} = 0
\end{equation}
which gives
\begin{equation}
\begin{aligned}
 \sum_{k=0}^{K-1}   \|{\bf{{{h}}}}_{k,T^c_0(k)}\|_{\ell_1}   \leq   \sum_{k=0}^{K-1}  \|{\bf{{{h}}}}_{k,T_0(k)}\|_{\ell_1}.
 \label{firststep}
 \end{aligned}
 \end{equation}
Considering  
\begin{equation}
\|{\bf{{{h}}}}_{k,T_0(k)}\|_{\ell_1} \leq s^{1/2} \|{\bf{{{h}}}}_{k,T_0(k)}\|_{\ell_2},
\end{equation}
and substituting it into (\ref{initial}) and (\ref{firststep}) we have 
\begin{equation}
\begin{aligned}
 \sum_{k=0}^{K-1} \|{\bf{{h}}}_{k,T_{01}^c(k)}\|_{\ell_2}   \leq    \sum_{k=0}^{K-1}  \|{\bf{{{h}}}}_{k,T_0(k)}\|_{\ell_2}.
\end{aligned}
\end{equation}

\section{Proof of  {{Lemma}} \ref{lemma3}} \label{AppD}
To find the bound for $ \sum_{k=0}^{K-1} \|{\bf{{h}}}_{k,T_{01}(k)}\|_{\ell_2}$, we start with 
\begin{equation}
{\bf{B}}{\bf{h}}_k = {\bf{B}}{\bf{h}}_{k,T_{01}(k)} + \sum_{j \geq 2} {\bf{B}}{\bf{h}}_{k,T_j(k)},
\end{equation}
which gives
\begin{equation}
\begin{aligned}
\| {\bf{B}}{\bf{h}}_{k,T_{01}(k)}  \|_{\ell_2}^2 &= \left\langle {\bf{B}}{\bf{h}}_{k,T_{01}(k)}, {\bf{B}}{\bf{h}}_k \right\rangle \\
&\quad- \langle{\bf{B}}{\bf{h}}_{k,T_{01}(k)}, \sum_{j \geq 2} {\bf{B}}{\bf{h}}_{k,T_j(k)} \rangle.
\end{aligned}
\end{equation}
From (\ref{main}) and the RIP condition for ${\bf{B}}$,
\begin{equation}
\begin{aligned}
|\langle{\bf{B}}{\bf{h}}_{k,T_{01}(k)}, {\bf{B}}{\bf{h}}_k \rangle|  &\leq  \| {\bf{B}}{\bf{h}}_{k,T_{01}(k)} \|_{\ell_2}   \| {\bf{B}}{\bf{h}}_k \|_{\ell_2} \\
& \leq 2 \epsilon C_0 \sqrt{1+ \delta_{2s}} \|{\bf{h}}_{k,T_{01}(k)}\|_{\ell_2},
\end{aligned}
\label{proof1}
\end{equation}
and, moreover, application of the parallelogram identity for disjoint subsets $T_0(k)$ and $T_j(k), j\geq 2$ results in
\begin{equation}
|\langle{\bf{B}}{\bf{h}}_{k,T_0(k)}, {\bf{B}}{\bf{h}}_{k,T_j(k)}\rangle|  \leq  \delta_{2s}  \| {\bf{h}}_{k,T_0(k)} \|_{\ell_2}  \| {\bf{h}}_{k,T_j(k)} \|_{\ell_2}.
\label{D4}
\end{equation}
Inequality (\ref{D4}) holds for $T_1$ in place of $T_0$. Since $T_0$ and $T_1$ are disjoint 
\begin{equation}
\| {\bf{h}}_{k,T_0(k)} \|_{\ell_2} + \| {\bf{h}}_{k,T_1(k)} \|_{\ell_2} \leq \sqrt{2} \| {\bf{h}}_{k,T_{01}(k)} \|_{\ell_2},
\end{equation}
which results in
\begin{equation}
\begin{aligned}
&(1-\delta_{2s}) \| {\bf{h}}_{k,T_{01}(k)} \|^2_{\ell_2} \leq \| {\bf{B}}{\bf{h}}_{k,T_{01}(k)} \|^2_{\ell_2}\\ 
&~~\leq  \| {\bf{h}}_{k,T_{01}(k)} \|_{\ell_2} ( 2  C_0 \epsilon \sqrt{1+ \delta_{2s}} + \sqrt{2} \delta_{2s}   \sum_{j \geq 2}  \| {\bf{h}}_{k,T_j(k)} \|_{\ell_2} ).
\end{aligned}
\label{D6}
\end{equation}
It follows from (\ref{othersets}) and (\ref{D6}) that
\begin{equation}
 \| {\bf{h}}_{k,T_{01}(k)} \|_{\ell_2} \leq \alpha  C_0 \epsilon  + \rho s^{-1/2} \| {\bf{h}}_{k,T_0^c(k)} \|_{\ell_2}.
 \label{D7}
\end{equation}
Now, using (\ref{firststep}) and (\ref{D7})  we can conclude that
\begin{equation}
\begin{aligned}
 \sum_{k=0}^{K-1} \|{\bf{{h}}}_{k,T_{01}(k)}\|_{\ell_2}  &\leq K \alpha  C_0 \epsilon + \rho s^{-1/2}  \sum_{k=0}^{K-1} {\bf{h}}_{k,T_0^c(k)} \|_{\ell_2} \\
&  \leq K \alpha  C_0 \epsilon + \rho  \sum_{k=0}^{K-1}  \|{\bf{{{h}}}}_{k,T_0(k)}\|_{\ell_2} \\
& \leq K \alpha  C_0 \epsilon + \rho  \sum_{k=0}^{K-1}  \|{\bf{{{h}}}}_{k,T_{01}(k)}\|_{\ell_2} , 
\end{aligned}
\label{proof2}
\end{equation}
which means
\begin{equation}
 \sum_{k=0}^{K-1} \|{\bf{{h}}}_{k,T_{01}(k)}\|_{\ell_2}  \leq K(1-\rho)^{-1}\alpha C_0 \epsilon.
\end{equation}

\section{Proof of Lemma \ref{UniquenessGeneralC}} \label{AppE}
We start the proof using contradiction. Let us assume that the sequence of $\left({\bf{u}}_k\right)_{k=0}^{K-1}$ and $\left({\bf{r}}_k\right)_{k=0}^{K}$ is not unique and there is another sequence of $s$-sparse $\left({\bf{\hat{u}}}_k\right)_{k=0}^{K-1}$ and $\left({\bf{\hat{r}}}_k\right)_{k=0}^{K}$ which satisfies the system  \eqref{linearDynamicFormulation} with noiseless measurements. Note that ${\bf{{u}}}_k-{\bf{\hat{u}}}_k$ has at most $2s$ nonzero elements. Similar to the depiction in Fig. \ref{SchemObservability}, we can rewrite \eqref{SeqObserv} based on $2s$ columns of ${\bf{B}}$ corresponding to $2s$ active non-zero elements of ${\bf{{u}}}_k-{\bf{\hat{u}}}_k$ as

\begin{equation}
\begin{pmatrix}
{\bf{y}}_0\\
{\bf{y}}_1\\
\vdots \\
{\bf{y}}_K
\end{pmatrix}
= \mathcal{O}_K {\bf{r}}_0 + \mathcal{J}_K^{2s}  
\begin{pmatrix}
{\bf{u}}_0^{2s}\\
{\bf{u}}_1^{2s}\\
\vdots \\
{\bf{u}}_{K-1}^{2s}
\end{pmatrix}
,
\end{equation}

\begin{equation}
\begin{pmatrix}
{\bf{y}}_0\\
{\bf{y}}_1\\
\vdots \\
{\bf{y}}_K
\end{pmatrix}
= \mathcal{O}_K {\bf{\hat{r}}}_0 + \mathcal{J}_K^{2s}  
\begin{pmatrix}
{\bf{\hat{u}}}_0^{2s}\\
{\bf{\hat{u}}}_1^{2s}\\
\vdots \\
{\bf{\hat{u}}}_{K-1}^{2s}
\end{pmatrix}
.
\end{equation}
By subtracting the above equations from each other we have
\begin{equation}
\mathcal{O}_K ({\bf{r}}_0 - {\bf{\hat{r}}}_0) + \mathcal{J}_K^{2s}  
\begin{pmatrix}
{\bf{u}}_0^{2s} - {\bf{\hat{u}}}_0^{2s}\\
{\bf{u}}_1^{2s} - {\bf{\hat{u}}}_1^{2s}\\
\vdots \\
{\bf{u}}_{K-1}^{2s} - {\bf{\hat{u}}}_{K-1}^{2s}
\end{pmatrix}
={\bf{0}}.
\label{SubtractionObserv}
\end{equation}
Based on assumptions that $rank(\mathcal{O}_K)=n$ and \mbox{$rank([\mathcal{O}_K~~\mathcal{J}_K^{2s}])=n+rank(\mathcal{J}_K^{2s}),~\forall \mathcal{J}_K^{2s} \in {\bf{J}}_K^{2s}$}, all columns of the observability matrix
must be linearly independent of each other, and of all columns of the $\mathcal{J}_K^{2s}$ matrix. Hence, the vector ${\bf{r}}_0 - {\bf{\hat{r}}}_0 = {\bf{0}}$. Having  ${\bf{r}}_0 = {\bf{\hat{r}}}_0$ and the matrix ${\bf{CB}}$ satisfying the RIP condition (\ref{RIP}) with isometry constant $\delta_{2s}<1$, it is easy to see that ${\bf{u}}_k={\bf{\hat{u}}}_k$ and therefore there exists unique state and $s$-sparse input sequences.

\section{Proof of Theorem \ref{TheoremNoisyObservability}} \label{AppF}
Lets assume that the the sequences $\left({\bf{{r}}}^*_k\right)_{k=0}^K$ and $s$-sparse $\left({\bf{{u}}}^*_k\right)_{k=0}^{K-1}$ are the solutions of Problem ($P2$). In this case \eqref{SubtractionObserv} can be rewritten as 
\begin{equation}
\mathcal{O}_K ({\bf{r}}^*_0 - {\bf{\bar{r}}}_0) + \mathcal{J}_K^{2s}  
\begin{pmatrix}
{\bf{u^*}}_0^{2s} - {\bf{\bar{u}}}_0^{2s}\\
{\bf{u^*}}_1^{2s} - {\bf{\bar{u}}}_1^{2s}\\
\vdots \\
{\bf{u^*}}_{K-1}^{2s} - {\bf{\bar{u}}}_{K-1}^{2s}
\end{pmatrix}
=\begin{pmatrix}
{\bf{e}}_0\\
{\bf{e}}_1\\
\vdots \\
{\bf{e}}_K
\end{pmatrix}
+\begin{pmatrix}
{\bf{e}}^*_0\\
{\bf{e}}^*_1\\
\vdots \\
{\bf{e}}^*_K
\end{pmatrix}
,
\end{equation}
where $\|{\bf{e}}^*_k \|_{\ell_2} \leq \epsilon$. Based on assumptions that \mbox{$rank(\mathcal{O}_K)=n$} and \mbox{$rank([\mathcal{O}_K~~\mathcal{J}_K^{2s}])=n+rank(\mathcal{J}_K^{2s}),~ \forall \mathcal{J}_K^{2s} \in {\bf{J}}_K^{2s}$} we can project the above equation using the projection \mbox{$({\bf{I}} - {\bf{P}}_{\mathcal{J}_K^{2s}})$} where ${\bf{P}}_{\mathcal{J}_K^{2s}}={\mathcal{J}_K^{2s}}({\mathcal{J}_K^{2s}}^T {\mathcal{J}_K^{2s}})^{-1}{\mathcal{J}_K^{2s}}^T$. It is straightforward to verify that $({\bf{I}} - {\bf{P}}_{\mathcal{J}_K^{2s}}){\mathcal{J}_K^{2s}} = {\bf{0}}$ and therefore there exists a $C_{\mathcal{J}}$ such that $\| {\bf{{r}}}^*_0  -  {\bf{{\bar{r}}}}_0 \|_{\ell_2} \leq C_{\mathcal{J}} \epsilon$. After finding the error bound for ${\bf{{r}}}^*_0$, sequentially we can find the error bound for the input vectors at each time. For instance at $k=1$, we have
\begin{equation}
\begin{aligned}
{\bf{y}}^*_1&={\bf{CAr}}^*_0+{\bf{CBu}}^*_0\\
{\bf{y}}_1&={\bf{CA\bar{r}}}_0+{\bf{CB\bar{u}}}_0+{\bf{e}}_0,~\|{\bf{e}}_0 \|_{\ell_2} \leq \epsilon\\
{\bf{y}}^*_1 &= {\bf{y}}_1+ {\bf{e}}^*_0,~ \|{\bf{e}}^*_0 \|_{\ell_2} \leq \epsilon
\end{aligned}
.
\end{equation}
Because the matrix ${\bf{CB}}$ satisfies the RIP condition (\ref{RIP}) with isometry constant $\delta_{2s}<\sqrt{2}-1$, with the same approach used in Appendices \ref{AppB} and \ref{AppD}, it is straightforward to verify that there exists a $C_k$ such that $\| {\bf{{u}}}^*_k  -  {\bf{{\bar{u}}}}_k \|_{\ell_2} \leq C_k \epsilon$, which means that always the recovered sparse input is upper bounded by a constant, $C_s$ multiple of the observation error.

\section*{Acknowledgments}
We would like to thank Professor Humberto Gonzalez (WUSTL) for helpful input and discussions. ShiNung Ching holds a Career Award at the Scientific Interface from the Burroughs-Wellcome Fund.  This work was partially supported by AFOSR 15RT0189, NSF ECCS 1509342 and NSF CMMI 1537015, from the US Air Force Office of Scientific Research and the US National Science Foundation, respectively.


%
%
%

\bibliographystyle{IEEEtran}
\bibliography{Ref}

\end{document}